\documentclass{amsart}
\usepackage[utf8]{inputenc}
\usepackage{amsmath,amsthm,amssymb}
\usepackage{hyperref}
\usepackage{enumerate}
\usepackage{xcolor}
\usepackage{colortbl}
\usepackage{array}
\usepackage{gensymb} 
\usepackage{bm}
\usepackage[margin=1in]{geometry}
\usepackage{multicol}
\usepackage{comment}
\usepackage{graphicx}
\usepackage{tikz}
\usepackage{tkz-graph}

\newtheorem{definition}{Definition}
\newtheorem{lemma}{Lemma}
\newtheorem{theorem}{Theorem}
\newtheorem{corollary}{Corollary}

\newtheorem{proposition}{Proposition}

\newtheorem{example}{Example}

\newcommand\sbullet[1][.5]{\mathbin{\vcenter{\hbox{\scalebox{#1}{$\bullet$}}}}}
\begin{document}

\title{Centrosymmetric Stochastic Matrices}
\author[Lei Cao]{Lei Cao \textsuperscript{1,2}}
\author[Darian McLaren]{Darian McLaren \textsuperscript{3}}
\author[Sarah Plosker]{Sarah Plosker \textsuperscript{3}}

\thanks{\textsuperscript{1}School of Mathematics and Statistics, Shandong Normal University, Shandong, 250358, China}
\thanks{\textsuperscript{2}Department of Mathematics, Halmos College, Nova Southeastern University, FL 33314, USA}
\thanks{\textsuperscript{3}Department of Mathematics and Computer Science, Brandon University,
Brandon, MB R7A 6A9, Canada}



\maketitle

\begin{abstract}
We consider the convex set $\Gamma_{m,n}$ of $m\times n$ stochastic matrices and the convex set $\Gamma_{m,n}^\pi\subset \Gamma_{m,n}$ of $m\times n$ centrosymmetric stochastic matrices (stochastic matrices that are symmetric under rotation by 180$\degree$). For $\Gamma_{m,n}$, we demonstrate a Birkhoff theorem for its extreme points  and  create a basis   from certain $(0,1)$-matrices.  For $\Gamma_{m,n}^\pi$, we characterize its extreme points   and create bases,  whose construction depends on the parity of $m$,  using our basis construction for stochastic matrices. For each of $\Gamma_{m,n}$ and  $\Gamma_{m,n}^\pi$, we further characterize their extreme points in terms of their associated bipartite graphs, we discuss a graph parameter called the fill and compute it for the various basis elements,  and we examine the number of vertices of the faces of these sets. We provide examples illustrating the results throughout.
\end{abstract}

\section{Introduction}

Doubly stochastic matrices are widely studied in the literature, in particular with regards to majorization theory. The extreme points of the set $\Omega_n$ of $n\times n$ doubly stochastic matrices have been characterized by permutation matrices via the infamous Birkhoff's theorem. A basis for $\Omega_n$ and faces of $\Omega_n$ are considered in \cite{BCD1967} and \cite{BG}, respectively. Various generalizations or relaxations of doubly stochastic matrices have also been studied, including a Birkhoff theorem for panstochastic matrices \cite{AlvisKinyon}, a geometric characterization of unistochastic matrices for small dimensions \cite{BengtssonEtAl}, extremal matrices of plane stochastic matrices of dimension 3 \cite{BrualdiCsisma}, line stochastic matrices of dimension 3 \cite{FischerSwart}, classes of substochastic matrices \cite{Fischer},  transportation polytopes \cite{ChoNam}, 
and  multistochastic tensors \cite{CuiLiNg, KeLiXiao}. Symmetric, Hankel-symmetric, and
centrosymmetric doubly stochastic matrices were recently considered in \cite{BC2018}.

We add to this body of literature by considering $m\times n$ stochastic matrices and $m\times n$ centrosymmetric stochastic matrices (symmetric or Hankel-symmetric stochastic matrices are in fact doubly stochastic, so they fall under the results in \cite{BC2018}). We  characterize the extreme points of the set of $m\times n$ stochastic matrices and the extreme points of the set of  $m\times n$ centrosymmetric stochastic matrices, and we create bases for these sets. We consider the   bipartite graphs associated to these matrices, giving alternate characterizations of the extreme points of these sets, and discuss a graph parameter called the fill. Finally, we describe the faces of these two sets.

Throughout, we let $m,n\in \mathbb{N}$ (positive integers) and  $[n]=\{1, \dots, n\}$. We use  $ M_{m,n}$ to denote the set of all $m\times n$ real-valued matrices, and $M_n$ when $m=n$.

\begin{definition}
A  matrix $A=(a_{i,j})\in M_{m,n}$ is \emph{stochastic} if $a_{i,j}\geq 0$ for all $i\in [m], j\in [n]$ and $\sum_{j=1}^na_{i,j}=1$ for all $i\in [m]$. A  matrix $A\in M_n$ is \emph{doubly stochastic} if $a_{i,j}\geq 0$ for all $i, j\in [n]$, $\sum_{j=1}^na_{i,j}=1$ for all $i\in [n]$, and  $\sum_{i=1}^na_{i,j}=1$ for all $j\in [n]$. 
\end{definition}
Note that the term stochastic refers to row stochasticity; one can consider the notion of column stochasticity by taking the transpose of a stochastic matrix. The convex set of all doubly stochastic matrices in $M_n$ will be denoted by $\Omega_n$ and the convex set of all stochastic matrices in $M_{m,n}$ by $\Gamma_{m,n}$; their corresponding sets of extreme points will be denoted $\mathcal E(\Omega_n)$ and $\mathcal E(\Gamma_{m,n})$, respectively.

For any $n\in \mathbb{N}$, let $\mathcal S_n$ be the set of permutations on $[n]$. Any permutation $\sigma\in \mathcal S_n$ is in one-to-one correspondence with a permutation matrix: a matrix $P=(p_{i,j})\in M_n$ such that $p_{i, \sigma(i)}=1$ for all $i\in [n]$ and $p_{i,j}=0$ for $j\neq \sigma(i)$. Put more simply, a permutation matrix is a $(0,1)$-matrix with exactly one 1 in each row and each column. For fixed $n$, the set of all permutation matrices is denoted $\mathcal{P}_n$.

\section{Extreme Points}
\subsection{Extreme Points of the Set of Stochastic Matrices}
The famous Birkhoff's theorem states that the set of doubly stochastic matrices $\Omega_n$ is the convex hull of the set $\mathcal{P}_n$ of all permutation matrices. It follows that the set of  extreme points $\mathcal E(\Omega_n)= \mathcal{P}_n$.

With Birkhoff's theorem for doubly stochastic matrices in mind, it is intuitively clear that $\mathcal E(\Gamma_{m,n})$, the set of all extreme points of the stochastic matrices, is precisely the $(0,1)$-matrices with exactly one 1 in each row. Many proofs of Birkhoff's theorem can easily be found; the analogue  for stochastic matrices is somewhat more elusive, but does appear in \cite[Lemma 1.2]{Gubin}. In the interest of being self-contained, we provide a proof for the stochastic case  below. The proof is similar to the proof of Birkhoff's theorem using Hall's marriage theorem (see \cite[Chapter 2]{MOA} and the references therein). 

We shall refer to the $(0,1)$-matrices in $M_{m,n}$  with exactly one 1 in each row as \emph{rectangular permutation matrices} (analogous results can be obtained for column stochastic matrices by considering $(0,1)$-matrices  with exactly one 1 in each column). We note that $m$ may or may not equal $n$ (when $m$ and $n$ are certainly equal, as in the case of doubly stochastic matrices, we revert back to the terminology of `permutation matrix' rather than `rectangular permutation matrix').
\begin{theorem}
The  extreme points of the set of $m\times n$ stochastic matrices are precisely the rectangular permutation matrices.
\end{theorem}

\begin{proof}
It suffices to show that every $m\times n$ stochastic matrix can be expressed as a convex combination of $m \times n$ rectangular permutation matrices. We proceed by finite induction on the number of nonzero elements of the matrix while keeping $m$ and $n$ fixed. From the row constraint on a stochastic matrix it is clear that the minimum number of nonzero elements is $m$,  in which case the stochastic matrix is in fact a rectangular permutation matrix. Now, consider a stochastic matrix $A_0$ with a number of nonzero elements greater than $m$. For each row, mark the smallest positive element (if this element is not unique,  mark the leftmost; this is done simply for the sake of making a choice and in fact one can randomly choose any positive element in each row). Let $a_0\in (0,1)$ denote the minimum of all of the marked elements and let $E_0$ denote the $m\times n$  $(0,1)$-matrix with a 1 in each of the marked positions. The matrix $A_1=(A_0-a_0E_0)/(1-E_0)$ is therefore an $m\times n$ stochastic matrix with fewer nonzero elements than $A_0$. By the induction hypothesis, there then exists coefficients $a_1,\dots,a_k\in (0,1]$ and  $(0,1)$-matrices with exactly one 1 in each row $E_1,\dots,E_k$ such that $A_1=\sum_{i=1}^k a_i E_i$, where $\sum_{i=1}^k a_i=1$. Expressing $A_0$ in term of $A_1$ and $E_0$ gives
\begin{eqnarray*}
A_0=a_0 E_0 + (1-a_0)A_1 = a_0 E_0 + (1-a_0)\sum_{i=1}^k a_i E_i,
\end{eqnarray*}
yielding $A_0$ as a linear combination of rectangular permutation matrices, with coefficients of the linear combination in the range of $(0,1)$. To see that this is in fact a convex combination,  we take the sum of the coefficients
\begin{equation*}
a_0 + (1-a_0)\sum_{i=1}^k a_i = a_0 + (1-a_0)(1)=1.
\end{equation*}
\end{proof}

We provide an example to illustrate the proof.
\begin{example}
Let $$\displaystyle A=\begin{pmatrix}1/2 & 0 & 1/2 & 0 \\ 7/10 & 0 & 0& 3/10 \\ 2/5 & 1/5 & 2/5 &0  \end{pmatrix}\in \Gamma_{3,4}.$$

To write $A$ as a convex combination of 
rectangular permutation matrices, one can do the following:

\begin{enumerate}
\item Randomly select a positive element from each row. Since the row sum is $1$ for each row, there must be positive entries in each row.

$A=\left(\begin{tabular}{cccc}\cellcolor[gray]{0.9}$1/2$ & $0$ & $1/2$ & $0$ \\ $3/10$ & $0$ & $0$& \cellcolor[gray]{0.9}$7/10$ \\ \cellcolor[gray]{0.9}$2/5$ & $1/5$ & $2/5$ &$0$  \end{tabular}\right)$

\item Subtract a $(0,1)$-matrix which has $1$'s at the same positions as selected positive entries multiplied by the least number among these three selected positive entries. Let
    $$E_0=\begin{pmatrix}1 & 0& 0& 0 \\ 0&0&0&1 \\ 1& 0& 0 &0 \end{pmatrix}$$
    $$A_1=\big(A- \frac{2}{5} E_0\big)/\big(1-\frac{2}{5}\big)=\begin{pmatrix}1/6 & 0 & 5/6 & 0 \\ 1/2 & 0 & 0 &1/2 \\ 0 &1/3 & 2/3 & 0  \end{pmatrix}.$$
    Then $A_1$ is in $\Gamma_{3,4}$ and $A_1$ has fewer positive elements than $A.$

\item Repeat this process. Let
$$E_1=\begin{pmatrix}0 & 0& 1& 0 \\ 1&0&0&0 \\ 0& 0& 1 &0 \end{pmatrix}$$
    $$A_2=\big(A_1- \frac{1}{2} E_1\big)/\big(1-\frac{1}{2}\big)=\begin{pmatrix}1/3 & 0 & 2/3 & 0 \\ 0 & 0 & 0 &1 \\ 0 &2/3 & 1/3 & 0  \end{pmatrix}.$$
\item
Let
$$E_2=\begin{pmatrix}0 & 0& 1& 0 \\ 0&0&0&1 \\ 0& 1& 0 &0 \end{pmatrix}$$
    $$A_3=\big(A_2- \frac{2}{3} E_2\big)/\big(1-\frac{2}{3}\big)=\begin{pmatrix}1 & 0 & 0  & 0 \\ 0 & 0 & 0 &1 \\ 0 & 0 & 1 & 0  \end{pmatrix}=E_3.$$

\end{enumerate}

Then we have,

\begin{eqnarray}\nonumber A&=&\frac{2}{5}E_0+\frac{3}{5}A_1\\
\nonumber &=&\frac{2}{5}E_0+\frac{3}{5}\big(\frac{1}{2}E_1+\frac{1}{2}A_2\big)\\
\nonumber &=& \frac{2}{5}E_0+\frac{3}{10}E_1+\frac{3}{10}A_2\\
\nonumber &=&\frac{2}{5}E_0+\frac{3}{10}E_1+\frac{3}{10}\big(\frac{2}{3}E_2+\frac{1}{3}E_3\big)\\
\nonumber &=&\frac{2}{5}E_0+\frac{3}{10}E_1+\frac{1}{5}E_2+\frac{1}{10}E_3.
\end{eqnarray}

\end{example}
Given a matrix $A\in M_n$, one can consider the usual definition of $A$ being symmetric if $A=A^t$, i.e.\ $a_{i,j}=a_{j,i}$ for all $i,j\in [n]$ (that is, $A$ is symmetric about its diagonal). Another type of symmetry is symmetry about the anti-diagonal: a matrix $A$ is \emph{Hankel symmetric} (also called persymmetric or mirror symmetric), denoted   $A=A^h$, if $a_{i,j}=a_{n+1-j, n+1-i}$ for all $i,j\in [n]$. A third type of symmetry is symmetry when rotated by 180$\degree$: a matrix $A\in M_{m,n}$ is \emph{centrosymmetric}, denoted $A=A^\pi$, if $a_{i,j}=a_{m+1-i, n+1-j}$ for all $i\in [m], j\in [n]$.
%
Note that if a matrix is either symmetric or Hankel symmetric, as well as stochastic, then the matrix is automatically  doubly  stochastic, so we focus on the study of  centrosymmetric stochastic matrices.

\subsection{Extreme Points of the Set of Centrosymmetric Stochastic Matrices}
We are interested in characterizing the extreme points of $\Gamma^\pi_{m,n}$, the set of all $m\times n$ centrosymmetric stochastic matrices. We will see that the parity of $m$ plays an important role.

\begin{lemma}\label{lemext} All extreme points of $\Gamma^\pi_{m,n}$ can be written in the form $\frac{1}{2}(R+R^\pi)$ where $R$ is a rectangular permutation matrix.
\end{lemma}

\begin{proof} Let $A\in \Gamma^\pi_{m,n}\subset \Gamma_{m,n}.$ Then \begin{equation}\label{eq1} A=c_1R_1+c_2R_2+\ldots+c_kR_k\end{equation} where $R_1,R_2,\ldots,R_k$ are rectangular permutation matrices, $c_i\geq 0$ and $c_1+c_2+\ldots+c_k=1.$ Then
\begin{equation}\label{eq2} A=A^\pi=c_1R_1^\pi+c_2R_2^\pi+\ldots+c_kR_k^\pi.\end{equation}
Take the average of \eqref{eq1} and \eqref{eq2},
  \begin{equation} \nonumber A=\frac{1}{2}(A+A^\pi)=\frac{1}{2}[c_1(R_1+R_1^\pi)]+\frac{1}{2}[c_2(R_2+R_2^\pi)]+\ldots+\frac{1}{2}[c_k(R_k+R_k^\pi)],\end{equation}
  which shows that any matrix in $\Gamma^\pi_{m,n}$ can be written as a convex combination of matrices in the form $\frac{1}{2}(R+R^\pi)$ for  rectangular permutation matrices $R.$
\end{proof}

\begin{lemma}\label{lemnotcent}
Let $R$ be a $m \times n$ rectangular permutation matrix where $m$ is an even number. If $R$ is not centrosymmetric then there exists $m \times n$ centrosymmetric rectangular permutation matrices $Q_1$ and $Q_2$ such that $R+R^\pi=Q_1+Q_2$ where $Q_1\neq Q_2$.
\end{lemma}
\begin{proof}
Let $\tilde{R}=R+R^\pi$ and so $\tilde{R}$ is a $m\times n$ centrosymmetric matrix where each row sums to 2. Let $\{\tilde{r}_i\}_{i=1}^m$ be the set of $m$ row vectors for $\tilde{R}$. From $R$ being a rectangular permutation matrix, each vector $\tilde{r}_i$ will either have a component with value 2, or two components with value 1. As $R$ is not centrosymmetric, there must exist a natural number $k\leq\frac{m}{2}$ such that $\tilde{r}_k$ satisfies the latter case. Either way, for every $i\in\{1,\dots,m\}$ there exists unit vectors $a_i$ and $b_i$ (that may or may not be equal) such that $\tilde{r}_i=a_i+b_i$. Define the $m\times n$ matrices $P=[p_i]_{i=1}^m$ and $Q=[q_i]_{i=1}^m$ and their respective row vectors $p_i$ and $q_i$ by
\begin{eqnarray*} p_i=\begin{cases}
      a_i & i\leq \frac{m}{2}\\
      0 & i>\frac{m}{2} \\
   \end{cases},\quad
  q_i=\begin{cases}
      b_i & i\leq \frac{m}{2}\\
      0 & i>\frac{m}{2} \\
   \end{cases}.
\end{eqnarray*}
We have that $P\neq Q$ (as there is at least one instance where $a_i\neq b_i$) and
\begin{eqnarray*}
\tilde{R}=(P+P^\pi)+(Q+Q^\pi).
\end{eqnarray*}
To see that this equality holds, note that by construction the first $\frac{m}{2}$ rows of $(P+P^\pi)+(Q+Q^\pi)$ must equal those of $\tilde{R}$ and so the equality of the remaining rows then follows from the centrosymmetry of both $(P+P^\pi)+(Q+Q^\pi)$ and $\tilde{R}$. The matrices $(P+P^\pi)$ and $(Q+Q^\pi)$ are $m\times n$ centrosymmetric rectangular permutation matrices and are not equal, which completes the proof.
\end{proof}

We illustrate the above lemma with the following examples.

\begin{example} Let $m=n=4$ and
$$R=\begin{pmatrix}1 & 0 & 0 & 0 \\ 1 & 0 & 0 & 0 \\ 0 & 1 & 0 & 0 \\ 0& 0 & 0 & 1\end{pmatrix}$$ which is not centrosymmetric. Then
\begin{eqnarray} \nonumber R+R^\pi&=&\begin{pmatrix}1 & 0 & 0 & 0 \\ 1 & 0 & 0 & 0 \\ 0 & 1 & 0 & 0 \\ 0& 0 & 0 & 1\end{pmatrix}+\begin{pmatrix}1 & 0 & 0 & 0 \\ 0 & 0 & 1& 0 \\ 0 & 0 & 0 & 1 \\ 0& 0 & 0 & 1\end{pmatrix}=\begin{pmatrix}2 & 0 & 0 & 0 \\ 1 & 0 & 1 & 0 \\ 0 & 1 & 0 & 1 \\ 0& 0 & 0 & 2\end{pmatrix}\\
\nonumber &=&\begin{pmatrix}1 & 0 & 0 & 0 \\ 0& 0 & 1 & 0 \\ 0 & 1 & 0 & 0 \\ 0& 0 & 0 & 1\end{pmatrix}+\begin{pmatrix}1 & 0 & 0 & 0 \\ 1 & 0 & 0 & 0 \\ 0 & 0 & 0 & 1 \\ 0& 0 & 0 & 1\end{pmatrix}
\end{eqnarray}
both of which are centrosymmetric stochastic matrices.

\end{example}

\begin{example} Let $m=4$ and $n=5$
$$R=\begin{pmatrix}1 & 0 & 0 & 0 & 0\\ 0& 1 & 0 & 0 & 0 \\ 0 & 1 & 0&0 & 0 \\ 0& 0 & 0 & 1&0\end{pmatrix}$$ which is not centrosymmetric. Then
\begin{eqnarray} \nonumber R+R^\pi&=&\begin{pmatrix}1 & 0 & 0 & 0 & 0\\ 0 & 1 & 0 & 0 & 0\\ 0 & 1 & 0 & 0 & 0\\ 0& 0 & 0 & 1& 0\end{pmatrix}+\begin{pmatrix}0& 1 & 0 & 0 & 0 \\ 0 & 0 & 0& 1& 0 \\ 0 & 0& 0 & 1 & 0 \\ 0& 0& 0 & 0 & 1\end{pmatrix}=\begin{pmatrix}1& 1 & 0 & 0 & 0 \\ 0& 1 & 0 & 1 & 0 \\ 0 & 1 & 0 & 1 & 0 \\ 0& 0 & 0 & 1& 1\end{pmatrix}\\
\nonumber &=&\begin{pmatrix}1 & 0 & 0 & 0 & 0\\ 0& 1 & 0 & 0 & 0 \\ 0 & 0 & 0 & 1 & 0 \\ 0& 0 & 0 & 0 & 1\end{pmatrix}+\begin{pmatrix}0&1 & 0 & 0 & 0 \\ 0 & 0 & 0 & 1 &0 \\ 0 & 1 & 0 & 0&0 \\ 0& 0 & 0 & 1& 0\end{pmatrix}
\end{eqnarray}
both of which are centrosymmetric  stochastic matrices.

\end{example}

Note that if $m$ is odd and $n$ is even, we may have extreme points containing $\frac{1}{2}.$ For example,
$$S=\begin{pmatrix}1&0 & 0 &0 \\ 0& \frac{1}{2} &\frac{1}{2} &0 \\ 0 & 0 & 0 &1\end{pmatrix}$$ is an extreme point of $\Gamma^\pi_{3,4}.$

\begin{lemma}\label{lemeven}
Let both $m$ and $n$ be positive integers with $m$  even ($n$ could be either even or odd). $A\in \Gamma^\pi_{m,n}$ is an extreme point if and only if $A$ is a centrosymmetric rectangular permutation matrix.
\end{lemma}

\begin{proof}
First assume  $A\in \Gamma^\pi_{m,n}$ is an extreme point. Then $A=\frac12(R+R^\pi)$ for some rectangular permutation matrix $R$ by Lemma \ref{lemext}. If $R$ is centrosymmetric, then $A=R$ and so $A$ is  a centrosymmetric rectangular permutation matrix. If $R$ is not centrosymmetric, then by Lemma \ref{lemnotcent}, $A=\frac12(R+R^\pi)=\frac12(Q_1+Q_2)$ where $Q_1$ and $Q_2$ are distinct centrosymmetric  rectangular permutation matrices. Thus $A$ can be written as a non-trivial convex combination of centrosymmetric stochastic matrices, so $A$ is not an extreme point, a contradiction. Thus $A$ is a centrosymmetric rectangular permutation matrix.

Conversely, any centrosymmetric rectangular permutation matrix $A$ must be an extreme point of $\Gamma^\pi_{m,n}$: since there is exactly one 1 in each row, the only way to decompose $A$ as a convex combination of matrices in $\Gamma^\pi_{m,n}$ is to have a single non-zero entry in each row in the exact same position as the 1 in each row of $A$. But, since the matrices are stochastic, a single non-zero entry in a row must be 1. Thus any decomposition into a convex combination is trivial.
\end{proof}

Lemma~\ref{lemext} gives the general decomposition for all extreme points of $\Gamma_{m,n}^\pi$;  therefore the following theorem characterizes the extreme points of $\Gamma_{m,n}^\pi$.
\begin{theorem}\label{thm:ExtCentroStoc}
Let $m$ and $n$ be positive integers. Let $R$ be an $m\times n$ rectangular permutation matrix. Then $\frac{1}{2}(R+R^\pi)$ is an extreme point of $\Gamma^\pi_{m,n}$ if and only if one of the following:
\begin{enumerate} [(i)]
\item $m$ is even and $R$ is a centrosymmetric rectangular permutation matrix;
\item $m$ is odd and $\tilde{R}$ is an $(m-1)\times n$ centrosymmetric rectangular permutation matrix, where $\tilde{R}$ is the rectangular permutation matrix obtained by removing the center row from $R.$
\end{enumerate}

\end{theorem}
\begin{proof}
(i) follows from Lemmas \ref{lemeven} and \ref{lemext}. Thus for $m$ even, $\frac{1}{2}({R}+{R}^\pi)$   is an extreme point of $\Gamma^\pi_{m,n}$ if and only if $R$ is a centrosymmetric rectangular permutation matrix.

For (ii), for $m$ odd, we need to establish that $\frac{1}{2}({R}+{R}^\pi)$   is an extreme point of $\Gamma^\pi_{m,n}$ if and only if $\tilde{R}$ is a centrosymmetric rectangular permutation matrix. Note that when $m$ is odd,  the entries of the center row of $\frac{1}{2}(R+R^\pi)$ can only be $0$, $1$, or $1/2$. Let ${A}=\frac{1}{2}({R}+{R}^\pi)$, $\tilde{A}=\frac{1}{2}(\tilde{R}+\tilde{R}^\pi)$ , $\vec{r}$ be the center row of $R$ and $\vec{a}$ be the center row of $A$. We proceed by proving the contrapositive of the forward direction. Assume that $\tilde{R}$ is not centrosymmetric. Since $\tilde{R}$ has an even number of rows, by Lemma \ref{lemnotcent} there exists $(m-1) \times n$ centrosymmetric rectangular permutation matrices $\tilde{Q}_1$ and $\tilde{Q}_2$ such that $\tilde{A}=\frac{1}{2}(\tilde{R}+\tilde{R}^\pi)=\frac{1}{2}(\tilde{Q}_1+\tilde{Q}_2)$ where $\tilde{Q}_1\neq \tilde{Q}_2$. Now, let $Q_1$ and $Q_2$ be the $m\times n$ centrosymmetric stochastic matrices, with respective center rows $\vec{r}$ and $(\vec{r})^\pi$, such that removing these center rows produces, respectively, the matrices $\tilde{Q}_1$ and $\tilde{Q}_2$. We then have $A=\frac{1}{2}({Q}_1+{Q}_2)$, which is clearly not an extreme point of $\Gamma^\pi_{m,n}$.

Conversely, assume that $\tilde{R}$ is centrosymmetric and hence, by (i), $\tilde{A}$ is an extreme point of $\Gamma^\pi_{(m-1),n}$. Consider now a decomposition of $A$ into a convex combination of linearly independent centrosymmetric stochastic matrices $\{B_1,B_2,\dots,B_z\}$: $A=\sum_{i=1}^z c_iB_i$ for $c_i> 0, \sum_{i=1}^z c_i=1$.   Deleting the center row of each of these matrices would then produce a decomposition of $\tilde{A}$, which must be trivial by part (i), and hence the matrices $\{B_1,B_2,\dots,B_z\}$ are identical up to their respective center rows. Let $\vec{b}_i$ for $i\in[z]$ be the center row for the matrix $B_i$. Since the maximum amount of non-zero entries of $\vec{a}$ is two and all the vectors $\vec{a},\vec{b}_1,\dots,\vec{b}_z$ are centrosymmetric, each vector $\vec{b}_i$ must have the same zero/nonzero pattern as $\vec{a}$. Moreover, since ${A}=\frac{1}{2}({R}+{R}^\pi)$ where $R$ is a rectangular permutation matrix, we note that there is a maximum of two non-zero entries in $\vec{a}$ that must sum to 1, and therefore the values for these entries will be uniquely determined. Hence, since $B_1, \dots, B_z$ are stochastic and have the same zero/nonzero pattern as $A$, it follows that  $\vec{a}=\vec{b}_i=\dots=\vec{b}_z$. Therefore $B_1=\cdots=B_z$ and so every decomposition of $A$ will be trivial.

\end{proof}

\section{Bases}
\subsection{A Basis for the Set of Stochastic Matrices}

The set of $m\times n$ stochastic matrices, $\Gamma_{m,n}$, is an affine space, but not a vector space.
For any matrix in $\Gamma_{m,n}$, given $n-1$ entries of a row, the remaining entry is fixed by the constraint of $\sum_{j=1}^n a_{i,j}=1$. Since this occurs for each of the $m$ rows, the dimension of $\Gamma_{m,n}$ is thus $m(n-1)$. Therefore, a basis for  $\Gamma_{m,n}$  contains $m(n-1)+1=mn-m+1$ linearly independent vectors.

For the set of all $n\times n$ doubly stochastic matrices, $\Omega_n$, the additional constraint of $\sum_{i=1}^n a_{i,j}=1$ fixes one entry  for each column as well as each row; the dimension of $\Omega_n$ is thus $(n-1)^2$, and a basis for $\Omega_n$ contains $(n-1)^2+1$ linearly independent vectors. A basis of $n\times n$ permutation matrices is given in \cite{BCD1967}; we outline the construction below, as we make use of this basis when creating a basis (of $n^2-n+1$ elements) for  $\Gamma_n$. This construction does not appear to generalize to $\Gamma_{m,n}$ (where $m\neq n$), as the process leads to too few linearly independent matrices. We provide an alternate construction of a basis $\Gamma_{m,n}$. We single out  the square case as it makes use of the technique in the literature for a construction of a basis for $\Omega_n$.

Following \cite{BCD1967}, we renumber the $(i,j)$-th position of an $l\times l$ matrix as the $[i+(j-i)l](\operatorname{mod} l^2)$-th position. For each integer $1\leq i\leq l^2$, let $A_i\in M_l$ be a $(0,1)$-matrix in which all elements are $0$ except the elements in positions $i, i+1, \dots, i+l-2(\operatorname{mod} l^2)$, which are all 1's. Each $A_i$ is thus ``almost'' a permutation matrix, except it has a row of all zeros. Now, for each integer $1\leq i\leq (n-1)^2$, let $P_i\in \mathcal P_n$ be the unique permutation matrix such that if one deletes the first row and first column  of $P_i$, the submatrix obtained is $A_i\in M_{n-1}$. A consequence of the analysis in \cite{BCD1967} is that $\{P_1, P_2, \dots, P_{(n-1)^2}\}$ is a linearly independent set in $M_n$.  Note that we will only use the renumbering of the positions of a matrix discussed in this paragraph in order to build the $P_i$.

Denote by $C_i \in \Gamma_{m,n}$ the $(0,1)$-matrix in which all elements are $0$ except the elements in the $i$-th column which are all $1$'s.

\begin{proposition}
The matrices
\begin{eqnarray*}
P_1,P_2,\dotsc,P_{(n-1)^2},C_1,C_2,\dotsc,C_n\in M_n
\end{eqnarray*}
are linearly independent.
\end{proposition}
\begin{proof}
Suppose that
\begin{eqnarray}\label{linDepRelation}
\sum_{i=1}^{(n-1)^2}\alpha_i P_i + \sum_{i=1}^n \beta_i C_i = O
\end{eqnarray}
where $\alpha_i$,$\beta_i\in\mathbb{R}$ and O is the $n\times n$ zero matrix. It suffices to show that $\alpha_1=\dotsb=\alpha_{(n-1)^2}=\beta_1=\dotsb=\beta_n=0$. To this end, for some $j\in[n]$ consider equation (\ref{linDepRelation}) in terms of the $j$-th column vector for each respective matrix. This now gives a sequence of $n$ relations of the form
\begin{eqnarray}\label{linDepRelationCol}
\sum_{i=1}^{(n-1)^2}\alpha_i (P_i)_j + \beta_j (C_j)_j = (O)_j
\end{eqnarray}
where $(\cdot)_j$ represents of the $j$-th column vector of the corresponding matrix. Note that the summation over the $C$'s is no longer needed as each $C$ matrix only has a single column with non-zero entries. Now, as the vector on the left hand side of equation (\ref{linDepRelationCol}) is the zero vector, the sum of its elements must also be 0. Hence the $n$ relations simplify to
\begin{equation*}
    \sum_{i=1}^{(n-1)^2} \alpha_i + n\beta_j=0.
\end{equation*}
By subtracting various pairs of these relations (for different $j$'s) we get $\beta_1=\beta_2=\dotsb=\beta_n$. But, as $C_1$ is the only matrix with a non-zero number in the $(1,1)$ entry we must have $\beta_1 =0$ and therefore $\beta_j=0$ for every $j\in[n]$. Equation (\ref{linDepRelation}) is simplified to
\begin{eqnarray*}
\sum_{i=1}^{(n-1)^2}\alpha_i P_i = O
\end{eqnarray*}
where it follows that $\alpha_i=0$ for every $i\in\{1,2,\dotsc,(n-1)^2\}$ as $P_1,P_2,\dotsc,P_{(n-1)^2}$ have already been shown to be linearly independent in \cite{BCD1967}.
\end{proof}

Since the matrices $P_1,P_2,\dotsc,P_{(n-1)^2},C_1,C_2,\dotsc,C_n$ are linearly independent stochastic matrices, and there are $n^2-n+1$ of them, the following result is immediate.
\begin{corollary}
The set $\{P_1,P_2,\dotsc,P_{(n-1)^2},C_1,C_2,\dotsc,C_n\}$ is  a basis for the set of stochastic $n\times n$ matrices $\Gamma_{n}$.
\end{corollary}

We now describe a method for constructing a basis for the more general setting of  $\Gamma_{m,n}$.
For each $i\in [m]$ and $j\in [n-1]$, denote by $B_{i,j}$  the $(0,1)$ matrix whose    $(i,j)$ entry  is $1$,  all entries in the $(j+1)$-th column are $1$ except the $(i,j+1)$ entry which is 0, and all other entries are 0.
\begin{example}
In $\Gamma_{5,3}$ there are 10 such $B$ matrices, which are listed below.
\[
B_{1,1}=\begin{pmatrix}1 & 0 & 0 \\ 0 & 1 & 0  \\ 0 & 1 & 0  \\ 0 & 1 & 0 \\ 0 & 1 & 0 \end{pmatrix}\quad
B_{2,1}=\begin{pmatrix}0 & 1 & 0 \\ 1 & 0 & 0  \\ 0 & 1 & 0  \\ 0 & 1 & 0 \\ 0 & 1 & 0 \end{pmatrix}\quad
B_{3,1}=\begin{pmatrix}0 & 1 & 0 \\ 0 & 1 & 0  \\ 1 & 0 & 0  \\ 0 & 1 & 0 \\ 0 & 1 & 0 \end{pmatrix}\quad
B_{4,1}=\begin{pmatrix}0 & 1 & 0 \\ 0 & 1 & 0  \\ 0 & 1 & 0  \\ 1 & 0 & 0 \\ 0 & 1 & 0 \end{pmatrix}\quad
B_{5,1}=\begin{pmatrix}0 & 1 & 0 \\ 0 & 1 & 0  \\ 0 & 1 & 0  \\ 0 & 1 & 0 \\ 1 & 0 & 0 \end{pmatrix}
\]\[
B_{1,2}=\begin{pmatrix}0 & 1 & 0 \\ 0 & 0 & 1  \\ 0 & 0 & 1  \\ 0 & 0 & 1 \\ 0 & 0 & 1 \end{pmatrix}\quad
B_{2,2}=\begin{pmatrix}0 & 0 & 1 \\ 0 & 1 & 0  \\ 0 & 0 & 1  \\ 0 & 0 & 1 \\ 0 & 0 & 1 \end{pmatrix}\quad
B_{3,2}=\begin{pmatrix}0 & 0 & 1 \\ 0 & 0 & 1  \\ 0 & 1 & 0  \\ 0 & 0 & 1 \\ 0 & 0 & 1 \end{pmatrix}\quad
B_{4,2}=\begin{pmatrix}0 & 0 & 1 \\ 0 & 0 & 1  \\ 0 & 0 & 1  \\ 0 & 1 & 0 \\ 0 & 0 & 1 \end{pmatrix}\quad
B_{5,2}=\begin{pmatrix}0 & 0 & 1 \\ 0 & 0 & 1  \\ 0 & 0 & 1  \\ 0 & 0 & 1 \\ 0 & 1 & 0 \end{pmatrix}
\]
\end{example}
\begin{theorem}\label{linDepOfB}
The matrices
\[B_{1,1},\dotsc,B_{m,(n-1)},C_n\in M_{m,n}\]
are linearly independent.
\end{theorem}
\begin{proof}
Suppose that
\begin{eqnarray}\label{linDepRelationB}
\sum_{i=1}^{m}\sum_{j=1}^{n-1}\beta_{i,j} B_{i,j} + \gamma_n C_n = O
\end{eqnarray}
where $\beta_{i,j}$,$\gamma_n\in\mathbb{R}$ and O is the $m\times (n-1)$ zero matrix. It suffices to show that $\beta_{1,1}=\dotsb=\beta_{m,(n-1)}=\gamma_n=0$. We proceed by strong finite induction on $j$. For $j=1$ we wish to show that $\beta_{1,1}=\beta_{2,1}=\dotsc=\beta_{m,1}=0$. By construction of the $B$ matrices, the only matrix in the collection $\{B_{1,1},\dotsc,B_{m,(n-1)},C_n\}$ with a $1$ in the $(1,1)$ entry is the matrix $B_{1,1}$. It follows immediately from equation (\ref{linDepRelationB}) that $\beta_{1,1}=0$. Similarly, by considering the matrices $B_{2,1},\dotsc,B_{m,1}$ we get that $\beta_{2,1}=\dotsb=\beta_{m,1}=0$. Now, let $k\in[n-2]$ and assume that the hypothesis holds for $j\in [k]$ (i.e.\ that $\beta_{1,1}=\beta_{2,1}=\dotsb=\beta_{m,1}=\beta_{1,2}=\beta_{2,2}=\dotsb=\beta_{m,2}=\dotsb=\beta_{m,k}=0$). Equation (\ref{linDepRelationB}) simplifies to
\begin{eqnarray}
\sum_{i=1}^{m}\sum_{j=k+1}^{n-1}\beta_{i,j} B_{i,j} + \gamma_n C_n = O.
\end{eqnarray}
By noting that $B_{1,(k+1)}$ is the only matrix in the collection $\{B_{(k+1),1},\dotsc,B_{m,(n-1)},C_n\}$ with a $1$ in the $(1, k+1)$ entry, $B_{2,(k+1)}$ is the only matrix with a $1$ in the $(2, k+1)$ entry, and so on, we get $\beta_{1,(k+1)}=\beta_{2,(k+1)}=\dotsb=\beta_{m,(k+1)}=0$. Therefore $\beta_{i,j}=0$ for every $i\in[m]$ and $j\in[n-1]$. In which case $\gamma_n$ must also be 0.
\end{proof}
 Since the matrices $B_{1,1},\dotsc,B_{m,(n-1)},C_n$ are linearly independent stochastic matrices, and there are $m(n-1)+1$ of them, the following result is immediate.
\begin{corollary}
The set $\{B_{1,1},\dotsc,B_{m,(n-1)},C_n\}$ is  a basis for the set of stochastic $m\times n$ matrices $\Gamma_{m,n}$.
\end{corollary}

\subsection{A Basis for the set of Centrosymmetric Stochastic Matrices}

A basis for the real vector space generated by the $n\times n$  centrosymmetric permutation matrices, for even $n$,  was given in \cite[Theorem 14]{BC2018}. The basis was built using a block structure. In light of this and the previous section, we provide a basis for the set of centrosymmetric stochastic matrices.

Recall we denote by $\Gamma^\pi_{m,n}$ the set of centrosymmetric $m\times n$ stochastic matrices. If $m$ is even, and so $m=2k$ for $k\in\mathbb{N}$, then $\Gamma^\pi_{m,n}$ has dimension $\frac{m}{2}(n-1)=k(n-1).$ Let $\mathcal{B}=\{B_1, B_2, \ldots, B_{k(n-1)+1}\}$ be a basis of $\Gamma_{k,n}$.

\begin{theorem}
If $m$ is even, then the collection $\hat{\mathcal{B}}=\{\hat{B}_1, \hat{B}_2,\ldots, \hat{B}_{k(n-1)+1}\}\subset M_{m,n}$, where
\[\hat{B}_i=\begin{pmatrix}B_i \\ B_i^\pi \end{pmatrix},\]
is a basis for $\Gamma^\pi_{m,n}$.
\end{theorem}
\begin{proof}
This follows immediately from the matrices in the collection $\mathcal{B}$ being linearly independent.
\end{proof}
On the other hand, consider if $m$ is odd, and so $m=2k+1$ for $k\in\mathbb{N}$.
The dimension of $\Gamma^\pi_{m,n}$ can be expressed as $(\frac{m-1}{2})(n-1)+\lceil\frac{n}{2}\rceil-1.$ There exists $l\in\mathbb{N}$ such that $n=2l$ or $n=2l+1$ depending on whether $n$ is, respectively, even or odd. In the case where $n$ is even the dimension of $\Gamma^\pi_{m,n}$  simplifies to $k(n-1)+l-1$, and for $n$ odd it simplifies to $k(n-1)+l.$ 
We recall $C_i$ are stochastic $(0,1)$-matrices having all entries zero except all ones in the $i$th column. In Theorem \ref{basisOddCentro} below, the $C_i$ are of size $k\times n$.


\begin{theorem}\label{basisOddCentro}
If $m$ is odd, the collection $\tilde{\mathcal{B}}=\{\tilde{B}_1, \tilde{B}_2,\ldots, \tilde{B}_{k(n-1)+1}, \tilde{C}_1, \tilde{C}_2,\dotsc, \tilde{C}_{\lceil\frac{n}{2}\rceil-1}\}\subset M_{m,n}$ is a basis for $\Gamma^\pi_{m,n}$, where the $\tilde{C}_i$ are defined to be
\[\tilde{C}_i=\begin{pmatrix}C_i \\ \vec{d}_i \\C_i^\pi \end{pmatrix},\]where $\vec{d}_i$ is an $n$-dimensional row vector with $\frac{1}{2}$ in the $i$ and $(n-i+1)$-th entries and all other entries 0. If $n$ is even, the $\tilde{B}_i$ are defined to be
\[\tilde{B}_i=\begin{pmatrix}B_i \\ \vec{d}_{\frac{n}2} \\B_i^\pi \end{pmatrix}.\]
 If instead $n$ is odd, then the $\tilde{B}_i$ are defined to be
\[\tilde{B}_i=\begin{pmatrix}B_i \\ \vec{e}_{\lceil\frac{n}2\rceil} \\B_i^\pi \end{pmatrix},\]
where $\vec{e}_{\lceil\frac{n}2\rceil} $ is the $n$-dimensional unit row vector with a 1 in the $\lceil\frac{n}2\rceil$-th entry (the $(l+1)$-th entry for $n=2l+1)$.
\end{theorem}
\begin{proof}
Regardless of whether $n$ is even or odd it follows that the $\tilde{B}_i$'s are linearly independent as the matrices in the collection $\mathcal{B}$ are linearly independent. Additionally, $\tilde{C}_1$ is the only matrix in $\tilde{\mathcal{B}}$ with a non-zero value in the $(k+1,1)$ entry, $\tilde{C}_2$ is the only matrix with a non-zero value in the $(k+1,2)$ entry, and so forth. Therefore all of the matrices in $\tilde{\mathcal{B}}$ are linearly independent and hence form a basis.
\end{proof}

\begin{example}\label{ex:basis}
Using a basis for $\Gamma_{2,4}$ as described in theorem \ref{linDepOfB} we get the following basis for $\Gamma^\pi_{5,4}$:
\[
\begin{pmatrix}1 & 0 & 0 & 0 \\ 0 & 1 & 0 & 0\\ 0 & 0.5 & 0.5 & 0  \\ 0 & 0 & 1 & 0 \\ 0 & 0 & 0 & 1 \end{pmatrix},\quad
\begin{pmatrix}0 & 1 & 0 & 0 \\ 1 & 0 & 0 & 0\\ 0 & 0.5 & 0.5 & 0  \\ 0 & 0 & 0 & 1 \\ 0 & 0 & 1 & 0 \end{pmatrix},\quad
\begin{pmatrix}0 & 1 & 0 & 0 \\ 0 & 0 & 1 & 0\\ 0 & 0.5 & 0.5 & 0  \\ 0 & 1 & 0 & 0 \\ 0 & 0 & 1 & 0 \end{pmatrix},\quad
\begin{pmatrix}0 & 0 & 1 & 0 \\ 0 & 1 & 0 & 0\\ 0 & 0.5 & 0.5 & 0  \\ 0 & 0 & 1 & 0 \\ 0 & 1 & 0 & 0 \end{pmatrix},\quad
\]\[
\begin{pmatrix}0 & 0 & 1 & 0 \\ 0 & 0 & 0 & 1\\ 0 & 0.5 & 0.5 & 0  \\ 1 & 0 & 0 & 0 \\ 0 & 1 & 0 & 0 \end{pmatrix},\quad
\begin{pmatrix}0 & 0 & 0 & 1 \\ 0 & 0 & 1 & 0\\ 0 & 0.5 & 0.5 & 0  \\ 0 & 1 & 0 & 0 \\ 1 & 0 & 0 & 0 \end{pmatrix},\quad
\begin{pmatrix}0 & 0 & 0 & 1 \\ 0 & 0 & 0 & 1\\ 0 & 0.5 & 0.5 & 0  \\ 1 & 0 & 0 & 0 \\ 1 & 0 & 0 & 0 \end{pmatrix},\quad
\begin{pmatrix}0 & 0 & 0 & 1 \\ 0 & 0 & 0 & 1\\ 0.5 & 0 & 0 & 0.5  \\ 1 & 0 & 0 & 0 \\ 1 & 0 & 0 & 0 \end{pmatrix}\quad
\]
\end{example}
\section{Graphs Associated to these Matrices}

Given a  matrix $A\in \Gamma_{m,n}$, consider the corresponding $(0,1)$-matrix $B$ having the same zero/non-zero pattern, and construct the bipartite graph associated to $B$ (such a graph has two vertex sets, one corresponding to the rows $R=\{r_1, \dots, r_m\}$ and the other corresponding to the columns $S=\{s_1, \dots, s_n\}$; an edge connects $r_i$ and $s_j$ if and only if $b_{i,j}=1$, i.e.\ $a_{i,j}\neq 0$); in this way, $B$ can be seen as the biadjacency matrix associated to the constructed bipartite graph. We  note that the bipartite graph corresponding to a   centrosymmetric stochastic matrix  $A\in \Gamma_{m,n}^\pi$ is centrosymmetric: there is an edge between vertices $r_i$ and $s_j$ if and only if there is an edge between vertices $r_{m+1-i}$ and $s_{n+1-j}$, for all $i\in [m], j\in[n]$.

The following is an immediate corollary of Theorem 1.
\begin{corollary}\label{cor:extStoc} $P$ is an extreme point of $\Gamma_{m,n}$ if and only if the bipartite graph associated with $P$ satisfies both of the following.
\begin{enumerate} [(i)]
\item It is a forest.
\item The degree of all row vertices is $1$ or, equivalently, all row vertices are leaves.
\end{enumerate}
\end{corollary}
Note that item (ii) of Corollary~\ref{cor:extStoc} implies that the sum of degrees of all column vertices  is $m$, and there is no path longer than 2.

For the extreme points of centrosymmetric stochastic matrices, we have the following immediate corollary of Theorem~\ref{thm:ExtCentroStoc}.
\begin{corollary}\label{cor:extCentroStoc} Let $m$ be even. $P$ is an extreme point of $\Gamma_{m,n}^\pi$ if and only if the bipartite graph associated with $P$ satisfies (i) and (ii) of Corollary~\ref{cor:extStoc} above. Let $m$ be odd. $P$ is an extreme point of $\Gamma_{m,n}^\pi$ if and only if the bipartite graph associated with $P$ satisfies both of the following.
\begin{enumerate} [(i)]
\item It is a forest.
\item The degree of the middle row vertex $r_{k+1}$ where $m=2k+1$ can be either 1 or 2; the degree of all other row vertices is $1$.
\end{enumerate}
\end{corollary}

Note that item (ii) of Corollary~\ref{cor:extCentroStoc} implies that the sum of degrees of all column vertices  is $m$ (if the degree of $r_{k+1}$ is 1) or $m+1$ (if the degree of $r_{k+1}$ is 2), and
there is no path longer than 4.

We illustrate Corollary~\ref{cor:extCentroStoc} with the following example.
\begin{example}
Consider the basis element and extreme point$$\begin{pmatrix}0 & 0 & 0 & 1 \\ 0 & 0 & 0 & 1\\ 0.5 & 0 & 0 & 0.5  \\ 1 & 0 & 0 & 0 \\ 1 & 0 & 0 & 0 \end{pmatrix}$$ of $\Gamma_{5,4}^\pi$ given in Example~\ref{ex:basis}. The corresponding bipartite graph is given below.

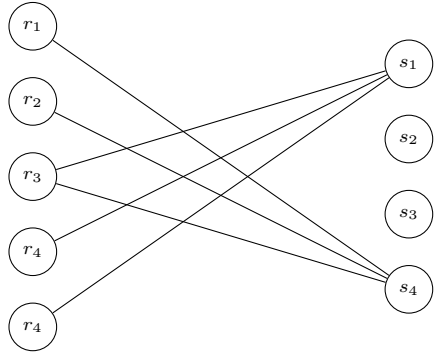
\begin{center}\begin{tikzpicture}
  [scale=1,auto=left,
  ns/.style={circle,fill=white!20,draw=black,minimum size=1mm},
  blank/.style={circle,fill=white!20,minimum size=1mm},
  es/.style={draw=black},
  dash/.style={draw=black}
   ]
  \node[ns] (r1) at (0,2) {\scriptsize{$r_1$}};
  \node[ns] (r2) at (0,1) {\scriptsize{$r_2$}};
  \node[ns] (r3) at (0,0) {\scriptsize{$r_3$}};
  \node[ns] (r4) at (0,-1) {\scriptsize{$r_4$}};
  \node[ns] (r5) at (0,-2) {\scriptsize{$r_4$}};
  \node[ns] (s1) at (5,1.5) {\scriptsize{$s_1$}};
    \node[ns] (s2) at (5,0.5) {\scriptsize{$s_2$}};
    \node[ns] (s3) at (5,-0.5) {\scriptsize{$s_3$}};
    \node[ns] (s4) at (5,-1.5) {\scriptsize{$s_4$}};

  \draw[es]  (r1) edge node[below]{} (s4);
  \draw[es]  (r2) edge node[above]{} (s4);
  \draw[es]  (r3) edge node[above]{} (s1);
  \draw[es]  (r3) edge node[above]{} (s4);
  \draw[es]  (r4) edge node[above]{} (s1);
  \draw[es]  (r5) edge node[above]{} (s1);

\end{tikzpicture}\end{center}
\end{example}

\medskip

The \emph{fill} of a graph is a graph parameter in (social) network analysis, giving some indication of the global connectivity of the graph. In particular, the fill of a graph is described as the probability that for two randomly chosen vertices, there is an edge between them; see \cite{KunegisKONECT}. It is defined  for general graphs as
$\displaystyle f(G)=\frac{E}{(|V||V-1|)/2}$ and
for bipartite graphs    $G=(V_1, V_2, E)$ as  $$f(G)=\frac{|E|}{|V_1|\, |V_2|},$$ where $|X|$ denotes the size of the set $X$ \cite{Kunegis}.  The fill of a graph also arises in older literature on random graph theory \cite{ErdosRenyi}, and more recently as an important topic in the realm of physics and society, as the intra-cluster density of a graph \cite{Fortunato}. Fill is also called the connectance of a graph, most often in the ecology literature when describing food networks; see, e.g.\ \cite{PoisotGravel, DFBG}.

Let $\lfloor \cdot\rfloor$ denote the floor function and $\lceil \cdot \rceil$ denote the ceiling function. The following can be readily verified.

\begin{proposition}
The fill of the bipartite graphs associated to the various matrices we've discussed is easily computed. We summarize in the table below.
\begin{center}
\bgroup
\def\arraystretch{1.5}
\begin{tabular}{|c|c|c|}\hline
Matrices&Space&Fill\\\hline
  $B_{i,j}$, $i\in [m]$, $j\in [n-1]$  & $M_{m,n}$&$\frac1n$ \\\hline
     $C_j$, $j\in [n]$ & $M_{m,n}$ &$\frac1n$
    \\ \hline
    $P_1, P_2, \dots, P_{(n-1)^2}$ &$M_n$&$\frac1n$\\ \hline
    $\hat{B}_1, \hat{B}_2, \dots, \hat{B}_{\frac{m}2(n-1)+1}$&$M_{m,n}$, $m$ even& $\frac1n$\\\hline
    $\tilde{C}_1, \tilde{C}_2, \dots, \tilde{C}_{\lceil\frac{n}{2}\rceil-1}$ &$M_{m, n}$& $ \frac{m+1}{mn}$\\\hline
    $\tilde{B}_1, \tilde{B}_2, \dots, \tilde{B}_{\lfloor\frac{m}{2}\rfloor(n-1)+1}$ &$ M_{m,n}$, $m$ odd, $n$ even& $\frac{m+1}{mn}$ \\\hline
     $\tilde{B}_1, \tilde{B}_2, \dots, \tilde{B}_{\lfloor\frac{m}{2}\rfloor(n-1)+1}$&$ M_{m,n}$, $m$ odd, $n$ odd& $\frac1n$ \\\hline
\end{tabular}
\egroup
\end{center}
 
\end{proposition}
In particular, we note that our basis elements for $\Gamma_{m,n}^\pi$ have larger fill when $m$ is odd and $n$ is even, versus the three other possible choices of parities ($\frac{m+1}{mn}>\frac1n$).

\section{The faces of $\Gamma_{m,n}$ and of $\Gamma_{m,n}^\pi$}
  The faces of $\Omega_n$ were characterized via the permanent function in \cite{BG}. The faces of $\Gamma_{m,n}$ can be found by letting $K$ be a subset of pairs of indices $(i,j)$, and constructing a $(0,1)$-matrix $B$ such that $b_{i,j}=1\Leftrightarrow (i,j)\in K$. Then the face of $\Gamma_{m,n}$ associated to the matrix $B$ is  $F(B)=\{X\in \Gamma_{m,n}\,|\, 0\leq x_{i,j}\leq b_{i,j}\,\forall i\in[m], j\in[n]\}$. In other words, the face $F(B)$ is the collection of all $m\times n$ stochastic matrices that are entry-wise less than or equal to $B$.
  The vertices of  $F(B)$ are precisely the rectangular  permutation matrices $P$ (that is, the extreme points of $\Gamma_{m,n}$), such that $p_{i,j}\leq b_{i,j}$ for all $i\in[m], j\in[n]$. To determine all non-empty faces of $\Gamma_{m,n}$ one can restrict to considering all
   $(0,1)$-matrices $B$ having \emph{row support}: that is, all $(0,1)$-matrices $B$ such that if $b_{r,s}=1$ for some $r\in[m], s\in[n]$, then there exists a rectangular permutation matrix $P$ such that $p_{r,s}=1$ and $p_{i,j}\leq b_{i,j}$ for all $i\in[m], j\in[n]$.  

   For any matrix $A\in M_{m,n}$,  for a fixed row $i\in[m]$, let $a_{i, \sbullet}=\sum_{j=1}^na_{i,j}$ be the row sum. The following proposition is immediate.

   \begin{proposition}
   The number of vertices of a face $F(B)$ of $\Gamma_{m,n}$, where $B$ is without loss of generality a $(0,1)$-matrix having row support, is $\Pi_{i\in [m]} b_{i, \sbullet}\,$.
   \end{proposition}

 \begin{example}
 Consider $B=\begin{pmatrix} 1 & 1 \\1 & 1\\0 & 1\end{pmatrix}$. Then there are four  rectangular permutation matrices that are entrywise less than or equal to $B$: $P_1=\begin{pmatrix} 1& 0\\1& 0\\0 & 1\end{pmatrix}$, $P_2=\begin{pmatrix} 1& 0\\0&1\\0&1\end{pmatrix}$, $P_3=\begin{pmatrix} 0 & 1\\1& 0\\0&1\end{pmatrix}$, $P_4=\begin{pmatrix}0&1\\0&1\\0&1 \end{pmatrix}$.
 \end{example}

When considering faces $F(B)$ of $\Gamma_{m,n}^\pi$ we consider all $m\times n$ centrosymmetric stochastic matrices which are entry-wise less than or equal to $B$.
To determine all non-empty faces of $\Gamma_{m,n}^\pi$ one can restrict to considering all
   $(0,1)$-matrices $B$ having a generalized notion of row support; in the case of centrosymmetric matrices, this refers to  all $(0,1)$-matrices $B$ such that if $b_{r,s}=1$ for some $r\in[m], s\in[n]$, then there exists a  matrix $C\in \mathcal E(\Gamma_{m,n}^\pi)$ such that $c_{r,s}=1$ and $c_{i,j}\leq b_{i,j}$ for all $i\in[m], j\in[n]$. The aforementioned extreme points $\mathcal E(\Gamma_{m,n}^\pi)$ in this case are those characterized in Theorem \ref{thm:ExtCentroStoc}.

   \begin{proposition}\label{prop:facescentrosym}
   The number of vertices of a face $F(B)$ of $\Gamma_{m,n}^\pi$, where $B$ is without loss of generality a $(0,1)$matrix having row support, is
   \begin{enumerate}[(i)]
   \item $\displaystyle \Pi_{i\in [\frac{m}{2}]} b_{i, \sbullet}\,$ if $m$ is even, and
   \item $\displaystyle \left\lceil\frac{b_{\left\lceil\frac{m}{2}\right\rceil, \sbullet}}2\right\rceil\Pi_{i\in [\frac{m-1}{2}]} b_{i, \sbullet}\, $ if $m$ is odd
   \end{enumerate}
   \end{proposition}

   While the proof of the above proposition is a simple counting argument, due to the complexity of the value for the $m$ odd case, we illustrate part (ii) in the following example.

   \begin{example}
   Consider $B=\begin{pmatrix}
   1 & 1 & 1\\1 & c& 1\\1 & 1 & 1
   \end{pmatrix}$, where either $c=0$ or $c=1$. We have $b_{1, \sbullet}=3$, and in the case where $c=1$, we have $b_{2, \sbullet}=2$. Proposition \ref{prop:facescentrosym} therefore states that the number of vertices is
   \[
   \displaystyle \left\lceil\frac{b_{2, \sbullet}}2\right\rceil b_{1, \sbullet}=\displaystyle \left\lceil\frac{2}2\right\rceil3=3.
   \]
   Working through the extreme points of $\Gamma_{3,3}^\pi$ we note that there are three that are entrywise less than or equal to B (as expected). They are:
   \[
   P_1=\begin{pmatrix} 1& 0 & 0\\0.5& 0& 0.5\\0 & 0&1\end{pmatrix}\, P_2=\begin{pmatrix} 0& 1 & 0\\0.5& 0& 0.5\\0 & 1&0\end{pmatrix}\,P_3=\begin{pmatrix} 0& 0 & 1\\0.5& 0& 0.5\\1 & 0&0\end{pmatrix}.
   \]
   Considering now the case where $c=1$, we have $b_{2, \sbullet}=3$ instead. Therefore Proposition \ref{prop:facescentrosym} now gives six vertices, which are listed below.
   \[
   \begin{pmatrix} 1& 0 & 0\\0.5& 0& 0.5\\0 & 0&1\end{pmatrix},\, \begin{pmatrix} 0& 1 & 0\\0.5& 0& 0.5\\0 & 1&0\end{pmatrix},\,\begin{pmatrix} 0& 0 & 1\\0.5& 0& 0.5\\1 & 0&0\end{pmatrix},\,
   \begin{pmatrix} 1& 0 & 0\\0& 1& 0\\0 & 0&1\end{pmatrix},\, \begin{pmatrix} 0& 1 & 0\\0& 1& 0\\0 & 1&0\end{pmatrix},\,\begin{pmatrix} 0& 0 & 1\\0& 1& 0\\1 & 0&0\end{pmatrix}.
   \]
   \end{example}
\section*{Acknowledgements}
 Sarah Plosker is supported by NSERC Discovery Grant number 1174582, the Canada Foundation for Innovation (CFI) grant number 35711, and the Canada Research Chairs (CRC) Program grant number 231250.

\end{document}